\documentclass[12pt]{amsart}

\usepackage[all]{xy}
\usepackage{graphics}
\usepackage{color}
\usepackage[T1]{fontenc}
\usepackage{parskip}
\usepackage[colorlinks]{hyperref}
\usepackage{enumitem}

\usepackage{tikz}
\usetikzlibrary{decorations.markings}

\usepackage[colorinlistoftodos,prependcaption]{todonotes}

\newcommand{\B}[1]{{\mathbf #1}}
\newcommand{\C}[1]{{\mathcal #1}}

\newtheorem{theorem}[subsection]{Theorem}
\newtheorem{corollary}[subsection]{Corollary}
\newtheorem{lemma}[subsection]{Lemma}
\newtheorem{proposition}[subsection]{Proposition}
\theoremstyle{definition}
\newtheorem{definition}[subsection]{Definition}

\newtheorem{example}[subsection]{Example}

\theoremstyle{remark}
\newtheorem{remark}[subsection]{Remark}

\numberwithin{figure}{section}
\numberwithin{table}{section}
\numberwithin{equation}{section}

\newcommand{\la}{\langle}
\newcommand{\ra}{\rangle}

\newcommand{\OP}{\operatorname}

\begin{document}

\title[On the autonomous norm on $\OP{Ham}(\B T^2)$]
{On the autonomous norm on the group of Hamiltonian
diffeomorphisms of the torus}
\author[Brandenbursky]{Michael Brandenbursky}
\address{Ben-Gurion University, Israel}
\email{brandens@math.bgu.ac.il}
\author[K\k{e}dra]{Jarek K\k{e}dra}
\address{University of Aberdeen and University of Szczecin}
\email{kedra@abdn.ac.uk}
\author[Shelukhin]{Egor Shelukhin}
\address{IAS, Princeton}
\email{egorshel@gmail.com}

\subjclass[2000]{Primary 53; Secondary 57}

\begin{abstract}
We prove that the autonomous norm on the group of Hamiltonian diffeomorphisms
of the two-dimensional torus is unbounded. We provide explicit examples
of Hamiltonian diffeomorphisms with arbitrarily large autonomous norm.
For the proofs we construct quasimorphisms on $\OP{Ham}(\B T^2)$ and some
of them are Calabi.
\end{abstract}

\maketitle

\setcounter{tocdepth}{1}
\tableofcontents

\section{Introduction}\label{S:intro}
Let $M$ be a smooth manifold and let $X\colon M\to TM$ be a compactly supported
vector field with the flow $\Psi_X\colon \B R\to \OP{Diff}(M)$.
The time-one map $\Psi_X(1)$ of the flow is called the
{\em autonomous} diffeomorphism associated with the vector field $X$. The
subset $\OP{Aut}(M)\subset \OP{Diff}_0(M)$ of autonomous diffeomorphisms is
conjugation invariant and, since the group of diffeomorphisms isotopic to the
identity is simple, it generates $\OP{Diff}_0(M)$.  In other words, a compactly
supported diffeomorphism of $M$ isotopic to the identity is a finite product of
autonomous ones. One may ask for a minimal decomposition and this question
leads to the concept of the autonomous norm which is defined by
$$
\|f\|_{\OP{Aut}} := \min\{n\in \B N\,|\,f = a_1\cdots a_n,\,a_i\in \OP{Aut}(M)\}.
$$
It is the word norm associated with the generating set $\OP{Aut}(M)$.
Since this set is conjugation invariant, so is the autonomous
norm. It follows from the work of Burago-Ivanov-Polterovich \cite{MR2509711}
and Tsuboi \cite{MR2523489,MR2874899} that for many manifolds all conjugation
invariant norms on $\OP{Diff}_0(M)$ are bounded. Hence the
autonomous norm is bounded in those cases.

The situation is different for the groups of area preserving diffeomorphisms of
surfaces. For example, the autonomous norm on the group
$\OP{Diff}_0(\B D^2,\OP{area})$ of compactly supported area preserving
diffeomorphisms of the open disc is unbounded \cite{MR3044593}. The same is
true for the group $\OP{Ham}(\Sigma)$ of Hamiltonian diffeomorphisms of
closed oriented surfaces different from the torus \cite{MR2104597, MR3391653,1405.7931}.
The present paper deals with the remaining case of the torus:

\begin{theorem}\label{T:main}
The autonomous norm on the group $\OP{Ham}(\B T^2)$ of Hamiltonian diffeomorphism
of the torus is unbounded.
\end{theorem}

One way to prove unboundedness of a conjugation invariant norm on a group $G$
is to construct an unbounded quasimorphism $\psi\colon G\to \B R$ which is
Lipschitz with respect to this norm. If such a norm is a word norm
then it suffices to construct a quasimorphism which is bounded on the generating
set which implies that it is Lipschitz. If $G=\OP{Diff}_0(M,\OP{vol})$ is
the group of volume preserving diffeomorphisms of a manifold $M$ then
nontrivial quasimorphisms on $G$ can be obtained from nontrivial quasimorphisms
on the fundamental group of $M$ as follows.

Let $z\in M$ be the basepoint and let ${\bf g}$ be an auxiliary Riemannian
metric on $M$. For every point $x\in M$ chose a path
$\gamma_x\colon [0,1]\to M$ from $z$ to $x$ by choosing a measurable
section of the map $\pi\colon P\to M$, where
$$
P=\{\gamma\colon [0,1]\to M\,|\, \gamma(0)=z,\,\gamma(1)=x \text{ and } \gamma
\text{ is a geodesic of } \bf g\}.
$$
Let $f\in \OP{Diff}_0(M,\OP{vol})$ and let $\{f_t\}$ be an isotopy from the
identity to~$f$. For every $x\in M$ the isotopy $\{f_t\}$ defines a loop
based at $x$ by
$\gamma(f,x) = \gamma_x \* f_t(x) \* \overline{\gamma_{f(x)}}$,
where the bar denotes the path in the reverse direction.
This loop is well defined up to homotopy of loops based at $z$ provided that evaluating loops of diffeomorphisms
of $M$ at the basepoint produces homotopically trivial loops in~$M$.
This holds, for example, if the center of the fundamental group of $M$
is trivial or if $\{f_t\}$ is a Hamiltonian isotopy in a symplectic manifold.

Let $\psi\colon \pi_1(M,z)\to \B R$ be a quasimorphism and let
$f\in \OP{Diff}_0(M,\OP{vol})$ be a compactly supported diffeomorphism isotopic
to the identity. Then, given that the volume of $M$ is finite, the map $\Psi\colon \OP{Diff}_0(M,\OP{vol})\to \B R$
defined by $$ \Psi(f) = \int_{M}\psi(\gamma(f,x)) dx $$ is a well defined
quasimorphism. This construction and the argument are due to Polterovich
\cite{MR2276956}. Notice that the construction can be performed for an action
$G\to \OP{Diff}_0(M,\OP{vol})$ of a group $G$ on $M$. For example, if $M$ is
simply connected then $G=\OP{Diff}_0(M)$ can act on another manifold which is
not simply connected. Concretely, if $\Sigma$ is a surface then
$\OP{Diff}_0(\Sigma,\OP{area})$ acts on the configuration space
$M=\B X_n(\Sigma)$. The fundamental group of this configuration space is (by
definition) the pure braid group on $n$-strings on the surface $\Sigma$.
Geometrically, this construction generalizes the above one in the sense that an
isotopy and a configuration of points defines a pure braid
$\gamma(f,x_1,\ldots, x_n)$ rather than a single loop (up to homotopy).
We provide more details in Section~\ref{SS:gg}.  Historically
the braid approach was the first original idea due to Gambaudo and Ghys
\cite{MR2104597} applied to diffeomorphisms of the disc and the sphere.  It was later
generalized by the first named author to other surfaces \cite{MR3391653}.
To sum up, the construction gives a linear map
$$
\C G\colon \B Q(\B P_n(\Sigma))\to \B Q(\OP{Ham}(\Sigma)),
$$
from the space of homogeneous quasimorphisms on the pure braid group to the
space of homogeneous quasimorphisms on the group of Hamiltonian diffeomorphisms
of the surface.

There are two main problems in proving the unboundedness of the autonomous
norm. The first, which is a general one, is to show that the above construction
yields nontrivial quasimorphisms. The second is to show that among these
nontrivial quasimorphisms there are ones which are bounded on the set of
autonomous diffeomorphisms.  These are the main objectives of the present paper
as well as the earlier ones~\cite{MR3391653,MR3044593,1405.7931}.

The solution of the first problem has two parts. The first one, which is
essentially the same for all surfaces, is the claim that the kernel of the
composition
$\B Q(\B B_n(\Sigma))\to \B Q(\B P_n(\Sigma)) \to \B Q(\OP{Ham}(\Sigma))$
consists of homomorphisms.  The idea of
the proof is due to Ishida who did it in the case of the disc and the sphere
\cite{MR3181631} and his argument was generalized to all surfaces in
\cite{MR3044593}. The second part is to construct nontrivial quasimorphisms on
the full braid group. Here, the solution depends on the genus.

The problem of identifying quasimorphisms on braid groups which yield
quasimorphisms vanishing on autonomous diffeomorphisms is the main
problem in all the cases and, again, solutions depend on the genus. In
what follows we provide a proof of Theorem \ref{T:main} by reducing the
argument to several results which are then proved in the rest of the paper.

\subsection*{Proof of Theorem \ref{T:main}}
The structure of the proof is presented in the following composition of linear maps.
$$
\xymatrix
{
\B Q(\B F_2;\B Z/2\times \B Z/2)\ar[r] &\B Q(\B F_2) \ar[r]^{\pi^*} &
\B Q(\B P_2(\B T^2))\ar[r]^{\C G} & \B Q(\OP{Ham}(\B T^2))\\
}
$$
Here, $\B Q(G)$ denotes a space of homogeneous quasimorphisms on a group~$G$
and $\B Q(\B F_2;\B Z/2\times \B Z/2)\subset \B Q(\B F_2)$ is the
subspace of quasimorphisms invariant under the action generated by
inverting generators.
\begin{itemize}[leftmargin=*]
\item
The construction of Gambaudo and Ghys (Section \ref{SS:gg}), provides a linear
map $\C G\colon \B Q(\B P_n(\B T^2))\to \B Q(\OP{Ham}(\B T^2))$ from the space
of homogeneous quasimorphisms of the pure braid group to the space of
homogeneous quasimorphisms of the group $\OP{Ham}(\B T^2)$ of Hamiltonian
diffeomorphisms of the torus. This map has a nontrivial kernel in general and
the goal is to construct a suitable quasimorphism $\psi$ on the pure braid
group such that its image $\C G(\psi)$ is a nontrivial quasimorphism bounded on
the set of autonomous diffeomorphisms.
\item
In our proof, we specify the braid group to two strings.
There is an isomorphism $\B P_2(\B T^2)\cong \B F_2\times \B Z^2$ (Lemma
\ref{L:P2(T)}).  We construct a suitable quasimorphism on the pure braid group
by constructing a quasimorphism $\psi\colon \B F_2\to \B R$ on the free group
and composing it with the projection $\pi\colon \B P_2(\B T^2)\to \B F_2$.
The free group here is the quotient of the braid group by its
center and hence the projection is canonical (i.e. every automorphism of $\B P_2(\B T^2)$ 
descends to an automorphism of the quotient $ \B F_2$).  Thus the refined goal is to
construct a quasimorphism $\psi\colon \B F_2\to \B R$ such that the image
$\C G(\psi\circ \pi)$ is nontrivial
and bounded on the set of autonomous elements.
\item
Let $\B F_2=\la a,b\ra$ and let $\sigma_a,\sigma_b\in \OP{Aut}(\B F_2)$
be automorphisms defined by
$
\sigma_a(a)=a^{-1},\, \sigma_a(b)=b,\,
\sigma_b(a)=a$ and  $\sigma_b(b)=b^{-1}$. They generate an action
of $\B Z/2\times \B Z/2$ on $\B F_2$. Let $\B Q(\B F_2;\B Z/2 \times \B Z/2)$
denote the space of homogeneous quasimorphisms which are invariant under
this action. The composition
$$
\B Q(\B F_2;\B Z/2 \times \B Z/2)\stackrel{\pi^*}\to
\B Q(\B P_2(\B T^2))\stackrel{\C G}\to \B Q(\OP{Ham}(\B T^2))
$$
is injective (Proposition \ref{P:injectivity}). Moreover, the space
$\B Q(\B F_2;\B Z/2 \times \B Z/2)$ is infinite dimensional
(Proposition \ref{P:Q-inf-dim}).
We obtain this way an infinite dimensional space of quasimorphisms
on the group $\OP{Ham}(\B T^2)$ and the next step is to prove
that it contains quasimorphisms bounded on autonomous diffeomorphisms.
\item
We prove in Lemma \ref{L:autonomous} that if a quasimorphism
$\psi\in \B Q(\B F_2)$ vanishes on primitive elements and on the commutator
$[a,b]$ of the generators then the quasimorphism $\C G(\pi^*\psi)$ vanishes on
autonomous elements. This reduces our task to showing that the space
$\B Q(\B F_2;\B Z/2 \times \B Z/2)$ contains quasimorphisms vanishing on
primitive elements and the commutator of the generators. Observe
that the second condition is automatic. Indeed, if
$\psi \in \B Q(\B F_2;\B Z/2 \times \B Z/2)$ then we have the
following computation in which we use invariance under $\sigma_a$
and homogeneity (which implies conjugation invariance).
\label{Eq:commutator}
\begin{align*}
\psi[a,b] &= \psi\left( aba^{-1}b^{-1} \right)\\
&=\psi\left( \sigma_a\left( a^{-1}bab^{-1} \right) \right)\\
&=\psi\left( a^{-1}bab^{-1} \right)\\
&=\psi\left( a^{-1}\left( bab^{-1}a^{-1} \right)a \right)\\
&=\psi[b,a] = -\psi[a,b].
\end{align*}
It follows that $\psi[a,b]=0$.
\item
Let $\sigma = \sigma_a\circ \sigma_b\in \OP{Aut}(\B F_2)$ be the automorphism
acting on a word by inverting all its letters. If a quasimorphism
$\psi\in \B Q(\B F_2)$ is invariant under the action of $\B Z/2\times \B Z/2$
then it is invariant under the action generated by $\sigma$. This, in turn,
implies that $\psi$ vanishes on palindromes (see the proof of Corollary
\ref{C:autonomous}).  It has been observed by Bardakov, Shpilrain and Tolstykh
\cite{MR2125453} that a primitive element of the free group $\B F_2$ is a
product of two palindromes. Hence every quasimorphism from
$\B Q(\B F_2;\B Z/2 \times \B Z/2)$ vanishes on primitive elements.  This
finishes the proof.  \qed
\end{itemize}

\begin{remark}
In the proof Proposition \ref{P:injectivity}  which claims
the injectivity of the homomorphism
$\B Q(\B F_2;\B Z/2 \times \B Z/2)\to \B Q(\OP{Ham}(\B T^2))$,
we provide explicit examples of Hamiltonian diffeomorphisms
on which quasimorphisms of the form $\C G(\pi^*\psi)$ evaluate
nontrivially (Example \ref{E:eggbeater}). Such examples are quite
standard and have been considered, for example, by Khanevsky
\cite{1409.7420v2} and Polterovich-Shelukhin \cite{MR3437837}.
\end{remark}

\begin{remark}
Another side result is concerned with the Calabi property and continuity
of the quasimorphisms we construct in the paper, see Section \ref{SS:snake} for a discussion of the Calabi property,
and a new example of a Calabi-type quasimorphism. More precisely, if
$\psi\in \B Q(\B F_2)$ is a nontrivial quasimorphism vanishing on palindromes then $\C G(\pi^*\psi)$ is nontrivial (see Corollary ~\ref{C:palindromes} 
and the discussion that follows it). Moreover, as proven in \cite{MR2253051} and \cite[Proposition 4.1]{MR2884036}, 
\begin{itemize}
\item if $\psi[a,b]\neq 0$ then $\C G(\pi^*\psi)$ has the Calabi property;
\item if $\psi[a,b]=0$ then $\C G(\pi^*\psi)$ is continuous in $C^0$-topology.
\end{itemize}

It follows that the quasimorphisms constructed in the proof of Proposition \ref{P:Q-inf-dim} are nontrivial and $C^0$-continuous. Since by Lemma \ref{L:autonomous} 
such quasimorphisms vanish on autonomous diffeomorphisms, we can streng\-then Theorem \ref{T:main} to the following statement:
\end{remark}

\begin{theorem}\label{P:C0-closure}
The group $\OP{Ham}(\B T^2)$ equipped with the word norm associated with the 
$C^0$-closure of the set of autonomous diffeomorphisms has infinite diameter.
\end{theorem}

\section{Preliminaries}\label{S:preliminaries}

In this section we provide necessary definitions, review in detail
the construction of Gambaudo-Ghys and state some known results which
we need for the proof.

\begin{definition}\label{D:norm}
Let $G$ be a group. A function $\|\cdot\|\colon G\to [0,\infty)$ is
called a {\em conjugation invariant norm} on $G$ if it satisfies the following conditions:
\begin{enumerate}
\item $\|f\|=0$ if and only if $f=1$,
\item $\left \|f^{-1}\right \|=\|f\|$,
\item $\|fg\|\leq \|f\|+\|g\|$,
\item $\left\|gfg^{-1}\right\|=\|f\|$.
\end{enumerate}
\end{definition}

\begin{definition}\label{D:qm}
A function $\psi\colon G\to \B R$ is called a {\em quasimorphism} if
there exist $D_{\psi}\geq 0$ such that the inequality
$$
\left|\psi(f) - \psi(fg) + \psi(g)\right| \leq D_{\psi}
$$
holds for all $f,g\in G$. A quasimorphism $\psi$ is called
homogeneous if
$$
\psi\left( f^n \right)=n\psi(f),
$$
for all $f\in G$ and $n\in \B Z$.
The space of all homogeneous quasimorphisms on a group $G$ is
denoted by $\B Q(G)$. Let ${\tt S}\subset G$. We denote by
$\B Q(G;{\tt S})$ the space of homogeneous quasimorphism which
vanish on ${\tt S}$.
\end{definition}

If $\psi\colon G\to \B R$ is a quasimorphism then its homogenization
$\overline{\psi}\colon G\to \B R,$ the unique homogeneous quasimorphism that differs from $\psi$ by a bounded function, satisfies
$$
\overline{\psi}(g)=\lim_{n\to \infty}\frac{\psi(g^n)}{n}.
$$
Moreover, the homogenization behaves well with respect to group actions.

\begin{lemma}\label{L:invariant-homogeneous}
Let $\psi\colon G\to \B R$ be a quasimorphism and let $\alpha\colon H\to\OP{Aut}(G)$
be an action of a group $H$ on the group $G$. If $\psi$ is invariant under
the action $\alpha$ then so is its homogenization.
\end{lemma}
\begin{proof}
The invariance of $\psi$ under the action $\alpha$ means that
$\psi(\alpha(h)(g)) = \psi(g)$ for every $h\in H$ and every $g\in G$.
The statement is a consequence of the following straightforward
computation.
\begin{align*}
\overline{\psi}(\alpha(h)(g)) &=\lim_{n\to \infty}\frac{\psi\left( (\alpha(h)(g))^n \right)}{n}\\
&= \lim_{n\to \infty}\frac{\psi\left(\alpha(h)(g^n)\right)}{n}\\
&= \lim_{n\to \infty}\frac{\psi\left(g^n\right)}{n}\\
&= \overline{\psi}(g).
\end{align*}
\end{proof}

\subsection{The Gambaudo-Ghys construction}\label{SS:gg}
Let $\Sigma$ be an oriented surface and let $\B X_n(\Sigma)$ denote the
space of configurations of ordered $n$-tuples of points in $\Sigma$.
Its quotient by the $n$-th symmetric group is the space of unordered
configurations and it is denoted by $\B C_n(\Sigma)$. The fundamental
groups
\begin{align*}
\B P_n(\Sigma)&:=\pi_1(\B X_n(\Sigma))\\
\B B_n(\Sigma)&:=\pi_1(\B C_n(\Sigma))
\end{align*}
are called the {\em pure braid group} and the (full) {\em braid group}
of the surface $\Sigma$ respectively.

Let $(z_1,\ldots,z_n)\in \B X_n(\Sigma)$ be an $n$-tuple of distinct points
which is the basepoint in the configuration space. We fix an auxiliary
Riemannian metric on $\Sigma$ and for every point
$x\in \Sigma$ we fix a geodesic $\gamma_{i,x}$ of minimal length from
$z_i$ to $x$. We denote by $\overline{\gamma}_{i,x}$ the reversed geodesic.

Let $h_t\in \OP{Ham}(\Sigma)$ be a Hamiltonian isotopy from the identity
to a diffeomorphism $h=h_1\in \OP{Ham}(\Sigma)$ and let
$(x_1,\ldots,x_n)\in \B X_n(\Sigma)$ be a point in the configuration space.
Let $\gamma(h,x_1,\ldots,x_n)\in \B P_n(\Sigma)$ be the braid
represented by the loop $[0,3]\to \B X_n(\Sigma)$ defined by
$$
s \mapsto
\begin{cases}
(\gamma_{1,x_1}(s),\ldots,\gamma_{n,x_n}(s)) &\text{ for } 0\leq s\leq 1\\
(h_{s-1}(x_1),\ldots,h_{s-1}(x_n)) &\text{ for } 1\leq s \leq 2\\
(\overline{\gamma}_{1,h(x_1)}(s-2),\ldots,\overline{\gamma}_{n,h(x_n)}(s-2))
&\text{ for } 2\leq s \leq 3.
\end{cases}
$$
This braid is only well defined on a set of points $(x_1,\ldots,x_n)$
of full measure.

Let $\psi\colon \B P_n(\Sigma)\to \B R$ be a homogeneous quasimorphism
and let
$$
\C G\colon \B Q(\B P_n(\Sigma))\to \B Q(\OP{Ham}(\Sigma))
$$
be defined by
$$
\C G(\psi)(h):=\lim_{p\to \infty}\frac{1}{p}
\int_{\B X_n(\Sigma)}\psi(\gamma(h^p,x_1,\ldots,x_n))\,dx_1\wedge \dots \wedge dx_n.
$$
The fact that the value $\C G(\psi)$ is a homogeneous quasimorphism, when
$\psi$ is a signature quasimorphism, was first proved by Gambaudo and Ghys
\cite{MR2104597} for the case of the disc and the sphere and later extended to
all $\psi$ and all surfaces by Brandenbursky \cite{MR3044593}. The map $\C G$
is linear and, in general, has a nontrivial kernel. In Section \ref{S:injectivity}
we prove that $\C G$ is injective on a certain subspace of
$\B Q(\B P_2(\B T^2))$.

\subsection{Braid groups on two strings}
We use the following presentations of the braid groups the free group:
\begin{align*}
\B B_2(\B T^2)&=\langle a_1,a_2,b_1,b_2,\sigma \,|\, \text{\tt Relations}\,\rangle\\
\B P_2(\B T^2)&=\langle a_1,a_2,b_1,b_2,\sigma^2 \,|\, \text{\tt Relations}\,\rangle\\
\B F_2 &=\langle a,b\rangle.
\end{align*}
We omit the relations because they are quite complicated and we don't need them
in our discussions. They can be found in \cite[Theorem 1.3 and 1.4]{MR0268889}.
The generators are presented in Figure \ref{F:generators}, which should be understood
as follows.
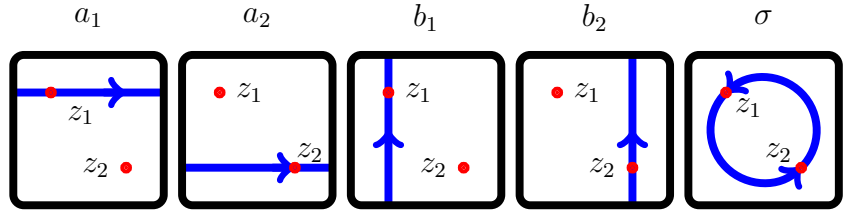
\begin{figure}[h]
\begin{tikzpicture}[line width=3pt, rounded corners, scale=0.1]
\draw[blue,->] (0,15) -- (15,15);
\draw[blue] (12,15) -- (20,15);
\filldraw[red]
(5,15) circle (7pt)
(15,5) circle (7pt);
\draw (9,12) node {$z_1$};
\draw (11,5) node {$z_2$};
\draw[rounded corners] (0,0) rectangle (20,20);
\draw (10,25) node {$a_1$};
\end{tikzpicture}
\begin{tikzpicture}[line width=3pt, rounded corners, scale=0.1]
\draw[blue,->] (0,5) -- (15,5);
\draw[blue] (12,5) -- (20,5);
\filldraw[red]
(5,15) circle (7pt)
(15,5) circle (7pt);
\draw (9,15) node {$z_1$};
\draw (17,7) node {$z_2$};
\draw[rounded corners] (0,0) rectangle (20,20);
\draw (10,25) node {$a_2$};
\end{tikzpicture}
\begin{tikzpicture}[line width=3pt, rounded corners, scale=0.1]
\draw[blue,->] (5,0) -- (5,10);
\draw[blue] (5,9) -- (5,20);
\filldraw[red]
(5,15) circle (7pt)
(15,5) circle (7pt);
\draw (9,15) node {$z_1$};
\draw (11,5) node {$z_2$};
\draw[rounded corners] (0,0) rectangle (20,20);
\draw (10,25) node {$b_1$};
\end{tikzpicture}
\begin{tikzpicture}[line width=3pt, rounded corners, scale=0.1]
\draw[blue,->] (15,0) -- (15,10);
\draw[blue] (15,9) -- (15,20);
\filldraw[red]
(5,15) circle (7pt)
(15,5) circle (7pt);
\draw (9,15) node {$z_1$};
\draw (11,5) node {$z_2$};
\draw[rounded corners] (0,0) rectangle (20,20);
\draw (10,25) node {$b_2$};
\end{tikzpicture}
\begin{tikzpicture}[line width=3pt, rounded corners, scale=0.1]
\draw[blue,->] (15,5) arc (-45:135:200pt);
\draw[blue,->] (5,15) arc (135:315:200pt);
\filldraw[red]
(5,15) circle (7pt)
(15,5) circle (7pt);
\draw (8,13) node {$z_1$};
\draw (12,7) node {$z_2$};
\draw[rounded corners] (0,0) rectangle (20,20);
\draw (10,25) node {$\sigma$};
\end{tikzpicture}
\caption{Generators of the braid group $\B B_2(\B T^2)$.}
\label{F:generators}
\end{figure}
For example, the generator $a_1$ is a braid in which the first
basepoint traces the horizontal loop going once around the torus and
the second basepoint remains still.

\begin{lemma}\label{L:P2(T)}
The map $\Phi\colon \B X_2(\B T^2)\to \B T^2\setminus\{0\}\times \B T^2$ defined
by $$
\Phi(x,y):=(x-y,y)
$$
is a diffeomorphism. It induces an isomorphism
$$
\Phi_*\colon \B P_2(\B T) \to \B F_2\times \B Z^2,
$$
which on the generators is given by
\begin{align*}
a_1&\mapsto (a,(0,0))\qquad
a_2\mapsto (a^{-1},(1,0))\qquad
\sigma^2 &\mapsto ([a,b],(0,0))\\
b_1&\mapsto (b,(0,0))\qquad
b_2\mapsto (b^{-1},(0,1)).
\end{align*}
\end{lemma}

\begin{proof}
The fact that $\Phi$ is a diffeomorphism is straightforward.
Let $\pi\colon \B P_2(\B T^2)\to \B F_2$ denote the projection onto
the free factor. The following figures describe the value of
$\pi$ on a generator.

\begin{figure}[h]
\begin{tikzpicture}[line width=3pt, rounded corners, scale=0.14]
\draw[blue,->] (5,0) -- (5,10);
\draw[blue] (5,9) -- (5,20);
\filldraw[red]
(5,15) circle (7pt)
(5,7) circle (5pt)
(15,5) circle (7pt);
\draw (12,15) node {$x=x_0$};
\draw (8,7) node {$x_t$};
\draw (17,5) node {$y$};
\draw[rounded corners] (0,0) rectangle (20,20);
\draw[|->] (25,10) -- (35,10);
\draw[green,->] (50,0) -- (50,15);
\draw[green] (50,9) -- (50,20);
\filldraw[red]
(50,10) circle (7pt)
(50,2) circle (5pt);
\draw (55,10) node {$x-y$};
\draw (55,2) node {$x_t-y$};
\draw[rounded corners] (40,0) rectangle (60,20);
\filldraw[white]
(40,0) circle (20pt);
\filldraw[white]
(40,20) circle (20pt);
\filldraw[white]
(60,0) circle (20pt);
\filldraw[white]
(60,20) circle (20pt);
\end{tikzpicture}
\caption{The image $\pi(b_1)=b\in \B F_2$.}
\label{F:a1}
\end{figure}
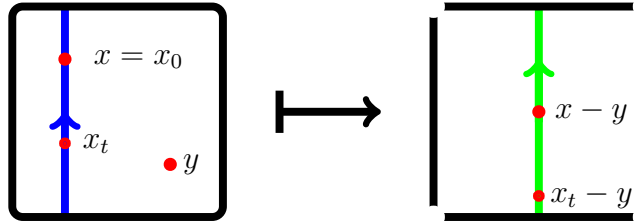

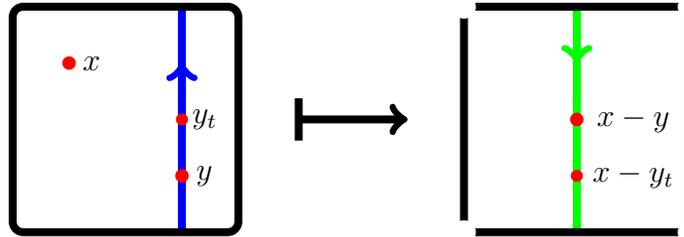
\begin{figure}[h]
\begin{tikzpicture}[line width=3pt, rounded corners, scale=0.15]
\draw[blue,->] (15,0) -- (15,15);
\draw[blue] (15,9) -- (15,20);
\filldraw[red]
(15,5) circle (7pt)
(15,10) circle (5pt)
(5,15) circle (7pt);
\draw (7,15) node {$x$};
\draw (17,10) node {$y_t$};
\draw (17,5) node {$y$};
\draw[rounded corners] (0,0) rectangle (20,20);
\draw[|->] (25,10) -- (35,10);
\draw[green] (50,0) -- (50,16);
\draw[green,<-] (50,15) -- (50,20);
\filldraw[red]
(50,10) circle (7pt)
(50,5) circle (5pt);
\draw (55,10) node {$x-y$};
\draw (55,5) node {$x-y_t$};
\draw[rounded corners] (40,0) rectangle (60,20);
\filldraw[white]
(40,0) circle (20pt);
\filldraw[white]
(40,20) circle (20pt);
\filldraw[white]
(60,0) circle (20pt);
\filldraw[white]
(60,20) circle (20pt);
\end{tikzpicture}
\caption{The image $\pi(b_2)=b^{-1}\in \B F_2$.}
\label{F:a2}
\end{figure}
The left hand side of the figure represents
the image of the generator $b_1$ (blue) with respect to the projection
onto the torus (the black square). The generator $b_1$ moves the
point $x$ along the meridian of the torus and keeps the point
$y$ fixed. The generator $b_2$ keeps the point $x$ fixed and moves
the point $y$ along the meridian of the torus. The right hand sides
of the figures present the free part of $\Phi_*(b_i)$ as loops
on the punctured torus. The abelian parts are straightforward to see.
The values on the generators $a_i$ are computed analogously.
\end{proof}

It follows from the above proposition that
the quotient of the pure braid group $\B P_2(\B T^2)$ by its center is isomorphic
to the free group $\B F_2$. Let
$\pi\colon \B P_2(\B T^2) \to \B F_2$
denote the projection. It induces the linear map
$$
\pi^*\colon \B Q(\B F_2)\to \B Q(\B P_2(\B T^2)).
$$
In the second part of the paper we will need quasimorphisms on the full braid
group. In what follows we identify those quasimorphisms $\psi$ on the free groups
such that $\pi^*\psi$ extends to the full braid group.

\begin{definition}\label{D:palindrome}
A word $w$ in $\B F_2=\langle a,b\rangle$ is called a {\em palindrome} if $w$
is equal to itself read from right to left. Let ${\tt PAL}\subset \B F_2$
denote the set of all palindromes.
\end{definition}

\begin{proposition}\label{P:palindromes}
A quasimorphism $\psi\in \B Q(\B F_2)$ vanishes on palindromes if and only if
the quasimorphism $\pi^*\psi$ extends to $\B B_2(\B T^2)$. In particular,
we get a linear map
$$
\B Q(\B F_2,{\tt PAL})\to \B Q(\B B_2(\B T^2)).
$$
\end{proposition}
\begin{proof}
The pure braid group is a normal subgroup of finite index in the full braid
group. According to \cite[Lemma 4.2]{MR2549454}, a homogeneous quasimorphism
$\psi\colon H\to \B R$ on a finite index normal subgroup $H\triangleleft G$
extends to the group $G$ if and only if for every $h\in H$ and every $g\in G$
we have that $\psi(ghg^{-1})=\psi(h)$.

It follows that a quasimorphism on the pure braid group extends
if and only if it is invariant under the automorphism defined by the
conjugation by $\sigma$. Since $\B F_2$ is the quotient of the pure
braid group by the center, the conjugation by $\sigma$ descends
to an automorphism of the free group. By abuse of notation we
denote it by $\sigma\in \OP{Aut}(\B F_2)$. Observe that $\sigma$
is defined by specifying its values on generators as
$\sigma(a)=a^{-1}$ and $\sigma(b)=b^{-1}$.

We conclude that
if $\psi \in \B Q(\B F_2)$ then the quasimorphism $\pi^*\psi$
extends to the full braid group if and only if $\psi$ is
invariant under $\sigma$. That is,
$\psi(\sigma(g))=\psi(g)$ for every $g\in \B F_2$.

Observe that $\sigma(g)=g^{-1}$ if and only if $g$ is a palindrome.
In particular, every element of the form $\sigma(g)g^{-1}$ is a
palindrome.

If $\psi$ vanishes on palindromes then $\psi(\sigma(g)g^{-1})=0$ for
every $g$. The following computation shows that $\psi$ is invariant
under $\sigma$. Let $w\in \B F_2$ be any element.
\begin{align*}
|\psi(\sigma(w))-\psi(w)|&=\frac{1}{n}|\psi(\sigma(w^n))-\psi(w^n)|\\
&\leq \frac{1}{n}\left(|\psi(\sigma(w^n)w^{-n})| + D_\psi\right)
=\frac{D_\psi}{n}.
\end{align*}

Conversely, if $\psi$ is invariant with respect to $\sigma$ then if
$g\in \B F_2$ is a palindrome we get that
$$
\psi(g^{-1})=\psi(\sigma(g)) = \psi(g)
$$
and by homogeneity we obtain that $\psi(g)=0$.
\end{proof}

\section{The injectivity theorem}\label{S:injectivity}
Let $\sigma_a,\sigma_b\in \OP{Aut}(\B F_2)$ be automorphisms
defined by
\begin{align*}
\sigma_a(a)&=a^{-1},\quad \sigma_a(b)=b\\
\sigma_b(a)&=a,\qquad \sigma_b(b)=b^{-1}.
\end{align*}
They generate an action of $\B Z/2\times \B Z/2$ on the
free group $\B F_2$.

\begin{proposition}\label{P:injectivity}
Let $\B Q(\B F_2;\B Z/2\times \B Z/2)\subset \B Q(\B F_2)$ be the space
of homogeneous quasimorphisms which are invariant under the
above action. The composition
$$
\B Q(\B F_2,\B Z/2\times \B Z/2)\stackrel{\pi^*}\to \B Q(\B P_2(\B T^2))
\stackrel{\C G}\to \B Q(\OP{Ham}(\B T^2))
$$
is injective.
\end{proposition}

\begin{lemma}\label{L:h(x)=x}
Let $h\in \OP{Ham}(\B T^2)$ be a Hamiltonian diffeomorphism and let
$x,y\in \B T^2$ be two points.  If $h(x)=x$ and $h(y)=y$ then
$$
\gamma(h^p,x,y)=\gamma(h,x,y)^p
$$
\end{lemma}

\begin{proof}
Immediate from the definition of $\gamma(-,-,-).$
\end{proof}

\begin{proof}
Let $\psi\in \B Q(\B F_2,\B Z/2\times \B Z/2)$. We shall prove
that $\C G(\pi^*\psi)\neq~0$ in $\B Q(\OP{Ham}(\B T^2))$
by constructing explicit examples
of Hamiltonian diffeomorphisms on which  $\C G(\pi^*\psi)$
evaluates nontrivially.

Let $s\in (0,\frac{1}{4})$ and let $0 <\epsilon <10^{-3}s$.
Let $F_s\colon [0,1]\to \B R$ be a smooth function with the following properties:
\begin{enumerate}
\item $F_s(x)=0$ for $x\in \left[0,\frac{1}{4}-s-\epsilon\right]\cup
\left[\frac{1}{4}+s+\epsilon,1\right]$,
\item $F_s'(x)=1$ for $x\in \left[\frac{1}{4}-s,\frac{1}{4}-\epsilon\right]$,
\item $F_s\left(\frac{1}{4}-x\right)=F_s\left(\frac{1}{4}+x\right)$
for $x\in \left[0,\frac{1}{4}\right]$,
\end{enumerate}
see Figure \ref{F:hamiltonian}.

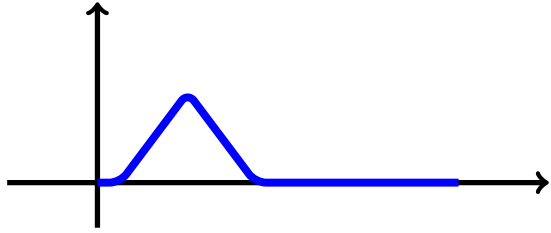
\begin{figure}[h]
\begin{tikzpicture}[line width=2pt, rounded corners, scale=0.3]
\draw[->] (-4,0) -- (20,0);
\draw[->] (0,-2) -- (0,8);
\draw[blue, line width=3pt] (0,0) -- (1,0) -- (4,4)  -- (7,0) -- (16,0);
\end{tikzpicture}
\caption{The function $F_s$.}
\label{F:hamiltonian}
\end{figure}

Let $H,V\colon \B T^2\to \B R$ be defined by
$H(x,y)=F_s(1-y)$ and $V(x,y)=F_s(x)$ respectively.
Let $h_t,v_t\in \OP{Ham}(\B T^2)$ be the corresponding Hamiltonian flows
and let $t_0>0$ be a real number chosen so that the restriction of
$v_{t_0}$ to the annulus $\left[\frac{1}{4}-s,\frac{1}{4}-\epsilon\right]\times \B S^1$
is the identity. Let $h:=h_{t_0}$ and $v:=v_{t_0}$.
The support of $v$ is marked green and the support of $h$ is marked
blue in Figure \ref{F:hv} below. The isotopy $\{v_t\}$ is supported
between the green lines and the support of the isotopy $\{h_t\}$ is
between the blue lines.

\begin{figure}[h]
\begin{tikzpicture}[line width=2pt, scale=0.5]
\draw[line width=.2pt, gray!20] (0,0) grid (16,16);
\draw[green] (2,0) -- (2,16);
\draw[green] (6,16) -- (6,0);
\draw (3,8) node {$\bf S_3$};
\draw[green, <-] (3,3) -- (3,7);
\draw[green, <-] (5,6) -- (5,2);
\draw[green, line width=1pt, dashed] (4,0) -- (4,16);
\draw[blue] (0,14) -- (16,14)  (16,10) -- (0,10);
\draw[blue, line width=1pt, dashed] (0,12) -- (16,12);
\draw (8,13) node {$\bf S_4$};
\draw[blue,->] (10,13) -- (14,13);
\draw[blue,->] (13,11) -- (9,11);
\draw[red] (2,14) -- (2,10) -- (6,10) -- (6,14) -- (2,14)
(2,12) -- (6,12) (4,10) -- (4,14);
\draw (3,13) node {$\bf S_1$};
\draw (8,8) node {$\bf S_2$};
\draw[line width=3pt] (0,0) -- (0,16) -- (16,16) -- (16,0) -- (0,0) -- (0,1);
\filldraw (4,12) circle (4pt) (12,4) circle (4pt);
\draw (4.5,12.5) node {$z_1$};
\draw (12.5,4.5) node {$z_2$};
\end{tikzpicture}
\caption{Diffeomorphisms $h$ and $v$.}
\label{F:hv}
\end{figure}
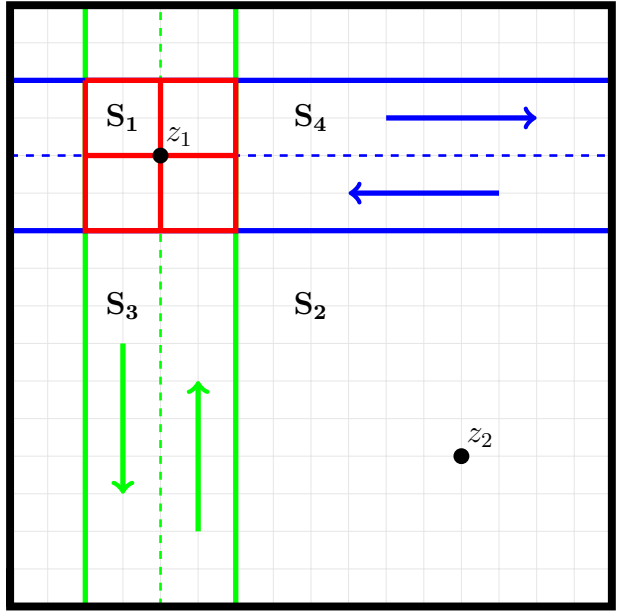

Define the following pairwise disjoint open subsets of the torus:
\begin{itemize}
\item
$S_1:=\left (\frac{1}{4}-s,\frac{1}{4}+s\right) \times
\left (\frac{3}{4}-s,\frac{3}{4}+s\right)$ -- the red square,
\item
$S_2:=\B T^2 \setminus \left ( \left[\frac{1}{4}-s,\frac{1}{4}+s\right]
\times [0,1] \cup [0,1]\times \left [\frac{3}{4}-s,\frac{3}{4}+s\right]\right )$
-- the complement of the union of the blue and green annuli,
\item
$S_3:= \left(\frac{1}{4}-s,\frac{1}{4}+s\right) \times [0,1]
\setminus \overline{S_1}$ -- the green annulus minus the closure of the red square,
\item
$S_4:= [0,1]\times \left(\frac{1}{4}-s,\frac{1}{4}+s\right)
\setminus \overline{S_1}$
-- the blue annulus minus the closure of the red square.
\end{itemize}

Let $\psi\colon \B F_2\to \B R$ be a nontrivial homogeneous quasimorphism
invariant under the action of $\B Z/2\times \B Z/2$ and
let $w(a,b)\in \B F_2$ be an element such that $\psi(w(a,b)) > 0$. Let
$g=w(v,h)\in \OP{Ham}(\B T^2)$.  Now we investigate the value of the integral
$$
\int_{\B X_2(\B T^2)}\psi(\pi(\gamma(g,x,y)))\,dx\wedge dy
$$
by decomposing it into a sum of integral over subsets of the
configuration space. First observe that the subset
$$
\bigcup_{i\neq j}S_i\times S_j \cup \bigcup_{i}\B X_2(S_i)\subset \B X_2(\B T^2)
$$
is open and dense so we have
\begin{align*}
\int_{\B X_2(\B T^2)}\psi(\pi(\gamma(g,x,y)))\,dx\wedge dy &=
\sum_{i\neq j} \int_{S_i\times S_j} \psi(\pi(\gamma(g,x,y)))\,dx\wedge dy\\
&+ \sum_i \int_{\B X_2(S_i)} \psi(\pi(\gamma(g,x,y)))\,dx\wedge dy
\end{align*}
It will be useful to know the volumes of the sets $S_i\times S_j$.
They are as follows:
\begin{itemize}[leftmargin=*]
\item
$\OP{vol}(S_1\times S_1)=16s^4$,
\item
$\OP{vol}(S_1\times S_2)=4s^2(1-2s)^2= 4s^2 - 16s^3 + 16s^4$,
\item
$\OP{vol}(S_1\times S_3)=\OP{vol}(S_1\times S_4)=4s^2(1-2s)2s=8s^3-16s^4$,
\item
$\OP{vol}(S_2\times S_2)=(1-2s)^2=1-4s+4s^2$,
\item
$\OP{vol}(S_2\times S_3)=\OP{vol}(S_2\times S_4)=(1-2s)^2(1-2s)2s=2s-12s^2+24s^3-16s^4$,
\item
$\OP{vol}(S_3\times S_3)=\OP{vol}(S_4\times S_4)
=\OP{vol}(S_3\times S_4)=4s^2(1-2s)^2=4s^2-16s^3+16s^4$.
\end{itemize}
The volumes are polynomials of $s$ and what will be important below is
their degrees.
Let us consider the element $\pi(\gamma(w(v,h),x,y))\in \B F_2$ for
various configurations:
\begin{itemize}[leftmargin=*]
\item
$(x,y)\in S_1\times S_2$; depending on the position of $x$ in the
red square we obtain:
\begin{itemize}
\item (top left) $\pi(\gamma(w(v,h),x,y))=w(a,b)$,
\item (top right) $\pi(\gamma(w(v,h),x,y))=w(a^{-1},b)$,
\item (bottom left) $\pi(\gamma(w(v,h),x,y))=w(a,b^{-1})$,
\item (bottom right) $\pi(\gamma(w(v,h),x,y))=w(a^{-1},b^{-1})$,
\end{itemize}
Since the quasimorphism $\psi$ is invariant under inverting generators
we get that $\psi(\pi(\gamma(g,x,y)))=\psi(w(a,b))\neq 0$. Thus
\begin{align*}
\int_{S_1\times S_2}\psi(\pi(\gamma(g,x,y)))\,dx\wedge dy &=
\OP{vol}(S_1\times S_2)\psi(w(a,b))\\
&=(4s^2 - 16s^3 + 16s^4)\psi(w(a,b)).
\end{align*}
\item
$(x,y)\in S_1\times (S_3\cup S_4)$; for fixed $g$ the braid $\gamma(g,x,y)$
can attain finitely many values in this case and we
let $C_{1}:=\max|\psi(\pi(\gamma(g,x,y)))|$.
\begin{align*}
\left|\int_{S_1\times S_3}\psi(\pi(\gamma(g,x,y)))\,dx\wedge dy\right|
&\leq \OP{vol}(S_1\times S_2)C_1\\
&=(6s^3-16s^4)C_1.
\end{align*}
\item
$(x,y)\in S_2\times (S_3\cup S_4)$; we get that $\pi(\gamma(g,x,y))$
is either a power of $a$ or a power of $b$ so
$\psi(\pi(\gamma(g,x,y)))=0$.
\item
$(x,y)\in S_3\times S_4$; here the situation is similar to the first case and the
value of $\pi(\gamma(g,x,y))$ depends on the positions $x$ and
$y$ in the strips and we obtain that
$\psi(\pi(\gamma(g,x,y)))=\psi(w(a,b))\neq 0$. We get
\begin{align*}
\int_{S_3\times S_4}\psi(\pi(\gamma(g,x,y)))\,dx\wedge dy &=
\OP{vol}(S_3\times S_4)\psi(w(a,b))\\
&=(4s^2 - 16s^3 + 16s^4)\psi(w(a,b)).
\end{align*}
\item
$(x,y)\in \B X_2(S_1)$; for fixed $g$ the braid $\gamma(g,x,y)$
can attain finitely many values and let
$C_2:=\max|\psi(\pi(\gamma(g,x,y)))|$. We have that
\begin{align*}
\left|\int_{\B X_2(S_1)}\psi(\pi(\gamma(g,x,y)))\,dx\wedge dy\right|
&\leq \OP{vol}(S_1\times S_1)C_1\\
&=16s^4C_2.
\end{align*}

\item
$(x,y)\in \B X_2(S_3)\cup \B X_2(S_4)$; in this case the
braid $\gamma(g,x,y)$ is equal to either $a_1^ma_2^n$ or
$b^m_1b^n_2$ and hence $\pi(\gamma(g,x,y))$ is equal to
a power of a generator and $\psi(\pi(\gamma(g,x,y)))=0$.

\end{itemize}

Chose $s\in \left(0,\frac{1}{4}\right)$ small enough so that
$$
(6s^3-16s^4)C_1 + 16s^4 C_2 < 2(4s^2 - 16s^3 + 16s^4) \psi(w(a,b)).
$$
For such an $s$ we obtain that
$$
\int_{\B X_2(\B T^2)}\psi(\pi(\gamma(g,x,y)))\,dx\wedge dy \neq 0.
$$
Since $g(x)=x$ and $g(y)=y$ for $(x,y)$ outside the subset of arbitrarily
small measure (depending on the number $\epsilon$), we have that
$\gamma(g^p,x,y)=\gamma(g,x,y)^p$, for $(x,y)$ in the set of measure
which is arbitrarily close to full, according to Lemma \ref{L:h(x)=x}.
This implies that $\psi(\pi(\gamma(g^p,x,y)))=p\psi(\pi(\gamma(g,x,y)))$
and finally we get that
$$
\lim_{p\to \infty}
\int_{\B X_2(\B T^2)}\frac{1}{p}\psi(\pi(\gamma(g^p,x,y)))\,dx\wedge dy \neq 0.
$$
which finishes the proof.
\end{proof}

\begin{example}[Eggbeater]\label{E:eggbeater}
Let $g = h^{2m}v^{2m-1}\cdots h^2v\in \OP{Ham}(\B T^2)$, where $h,v\in \OP{Ham}(\B T^2)$
are Hamiltonian diffeomorphisms defined in the above proof. It follows from
Proposition \ref{P:Q-inf-dim} below and the above proof that the quasimorphism
$\C G(\psi_{c_m}\circ \pi)$ is unbounded on the cyclic subgroup
of $\OP{Ham}(\B T^2)$ generated by $g$. Here
$\psi_{c_m}\colon \B F_2\to \B R$ is the quasimorphism associated
with the function $c_m\colon \B Z^m\to \B Z$ given by
$$
c_m(i_1,\ldots,i_m)=\OP{sgn}(|i_1|-|i_m|),
$$
see Lemma \ref{L:qm}. Since the quasimorphism $\C G(\psi_{c_m}\circ \pi)$ vanishes
on autonomous elements, we get that the cyclic subgroup generated
by $g$ is unbounded with respect to the autonomous norm.

\hfill $\diamondsuit$
\end{example}

\section{Quasimorphisms with vanishing properties}
\label{S:quasimorphisms}
In this section we prove that the space $\B Q(\B F_2;\B Z/2\times \B Z/2)$
of quasimorphisms on the free group invariant under the action of the
Klein group is infinite dimensional. It is done by constructing
explicit examples. Our construction is inspired by the example from the
proof of Theorem 1.1 in \cite{MR2125453}.

Let $w\in \B F_n$ be a reduced word.
A {\em syllable} in $w$ is a maximal power of a generator occurring
in $w$. The exponent of a syllable $s$ is denoted by $e(s)$. For example,
the commutator $[a,b]\in \B F_2$ has four syllables all with exponents
equal to $1$.
Let $w=s_1s_2\dots s_k$ be a reduced word, where $s_i$ are syllables.
It defines a $k$-tuple of integers $(e(s_1),e(s_2),\ldots,e(s_k))$.

Let $c\colon \B Z^m\to \B Z$ be a bounded function which satisfies
the identity:
\begin{equation}\label{Eq:inverse}
c(i_1,\ldots,i_m) = - c(-i_m,\ldots,-i_1).
\end{equation}
Let $\psi_c\colon \B F_n\to \B R$ be defined as follows.
Let $w=s_1\dots s_k\in \B F_n$. If $k<m$ then $\psi_c(w)=0$.
If $k\geq m$ then
$$
\psi_c(s_1s_2\dots s_k):= \sum_{i=1}^{k-m+1}c(e(s_i),\ldots,e(s_{i+m-1})).
$$

\begin{lemma}\label{L:qm}
Let $c\colon \B Z^m\to \B Z \cap [-B,B]$ be a bounded function satisfying
the identity (\ref{Eq:inverse}).
Then the function $\psi_c$ is a quasimorphism with defect bounded
by $3(m+1)B$.
\end{lemma}
\begin{proof}
Let $s=s_1\cdots s_k$ and $t=t_1\cdots t_l$ be reduced words
such that $st$ is also reduced. We have the following expression
for the value of $\psi_c(st)$:
\begin{align*}
\psi_c(st)&=
\begin{cases}
\psi_c(s)+\psi_c(t)\\
+ \sum_{i=1}^{m-1}c(e(s_{k-m+1+i}),\ldots,e(s_{k}),e(t_1),\ldots,e(t_i))\\
\text{ if last letter of $s$ is different from the first letter of $t$, or}\\
\psi_c(s)+\psi_c(t)\\
+ \sum_{i=1}^{m-1}c(e(s_{k-m+1+i}),\ldots,e(s_{k}t_1),\ldots,e(t_i))\\
- c(e(s_{k-m+1}),\ldots,e(s_k)) - c(e(t_1),\ldots,e(t_m)),\\
\text{ otherwise.}
\end{cases}
\end{align*}
It follows that in this special case we have the following estimate:
$$
\left|\psi_c(s) - \psi_c(st) + \psi_c(t)\right| \leq (m+1)B.
$$

Let us now consider the general case.
Let $s=s_1\cdots s_k,\, t=t_1\cdots t_l, u=u_1\cdots u_p\in \B F_n$ be reduced words,
where $s_i,t_i$ and $u_i$ are syllables. Suppose that $su$, $u^{-1}t$ and $st$
are reduced. Using the previous inequality we obtain the following estimate of the defect:
\begin{align*}
&\left|\psi_c(su)-\psi_c(st)+\psi_c(u^{-1}t)\right| \leq\\
&\leq \left| \psi_c(s)+\psi_c(u) - \psi_c(s) - \psi_c(t) +\psi_c(u^{-1}) + \psi_c(t)\right | + 3(m+1)B\\
&=\left|\psi_c(u) + \psi_c(u^{-1}) \right | + 3(m+1)B\\
&\leq 3(m+1)B.
\end{align*}
The fact that $\psi_c\left( u^{-1} \right)=-\psi_c(u)$ follows from the
identity (\ref{Eq:inverse}).  This proves that $\psi_c$ is a quasimorphism with
defect $3(m+1)B$.
\end{proof}

\begin{example}\label{E:bst}
Let $c\colon \B Z^2\to \B Z$ be defined by $c(m,n)=\OP{sgn}(|m|-|n|)$.
The associated quasimorphism $\psi_c\colon \B F_2\to \B Z$ is clearly
invariant under the action of $\B Z/2\times \B Z/2$. To see that
it is unbounded consider the cyclic subgroup generated by
$a^4b^3a^2b$. We have that
$$
\psi_c\left(\left(a^4b^3a^2b)^n\right)\right)=2n+1.
$$
which implies that the homogenization of this quasimorphism is nontrivial.
It follows from Lemma \ref{L:invariant-homogeneous} that the homogenization
is also invariant.
\end{example}

\begin{proposition}\label{P:Q-inf-dim}
The subspace $\B Q(\B F_2;\B Z/2\times \B Z/2)\subset \B Q(\B F_2)$
of homogeneous quasimorphism invariant under the action of
$\B Z/2\times \B Z/2$ is infinite dimensional.
\end{proposition}
\begin{proof}
Let $c_m\colon \B Z^m\to \B Z$ for $m\geq 2$ be a function defined by
$$
c_m(i_1,\ldots,i_m):= \OP{sgn}(|i_1|-|i_m|).
$$
Consider the sequence $\psi_{c_m}$ of quasimorphisms
defined in the beginning of this section. Since the function
$c_m$ depends only on the absolute values, the quasimorphisms
$\psi_{c_m}$ are invariant under inverting generators.

Let $w_m = a^{2m}b^{2m-1}\cdots a^2b$. We get that
$$
\psi_{c_m}\left( w^n \right) = 2n+m-1.
$$
Let $k\in \B N$ be a positive integer.
Consider the square $k\times k$-matrix with entries $a_{ij}=\psi_{c_{3i}}(w_{3j})$
and observe that it is upper triangular with positive entries on the diagonal
and hence it has a positive determinant. It implies that the
functions $\psi_{c_{3i}}$ for $i=1,\ldots,k$ are linearly independent
for any $k\in \B N$. This shows that there exists an infinite dimensional
subspace of quasimorphisms invariant under inverting generators.
It follows from Lemma \ref{L:invariant-homogeneous}
that $\dim \B Q(\B F_2;\B Z/2\times \B Z/2)= \infty$.
\end{proof}

\section{Vanishing on autonomous diffeomorphisms}\label{S:vanishing}
\begin{lemma}\label{L:autonomous}
Let $\psi\colon \B F_2=\langle a,b\rangle\to \B R$ be a homogeneous quasimorphism. If it
vanishes on primitive elements and on the commutator $[a,b]$
then the quasimorphism
$$
\C G(\psi \circ \pi)\colon\OP{Ham}(\B T^2)\to \B R
$$
vanishes on the set of autonomous diffeomorphisms.
\end{lemma}
\begin{proof}
Let $H\colon \B T^2\to \B R$ be a smooth function and let
$h_t\in \OP{Ham}(\B T^2)$ be the autonomous flow induced by $H$.
Let
$$
\C O(h_t,x):=\{h_t(x)\in \B T^2\,|\,t\in \B R\,\}
$$
denote the orbit of the point $x$ with respect to the flow $h_t$.
Such an orbit is either periodic (including constant)
or it is an interval between one (homoclinic) or two (heteroclinic)
fixed points.

Let $z_1,z_2\in \B T^2$ be basepoints and let $x,y\in \B T^2$.
In what follows we analyze the braid $\gamma(h^p,x,y)$ for
$p\in \B N$.
We break it down into cases depending on the form of the orbits
$\C O(h_t,x)$ and $\C O(h_t,y)$.
We consider the following cases:
\begin{enumerate}[leftmargin=*]
\item $\C O(h_t,x)=\{x\}$.

\begin{enumerate}
\item If $\C O(h_t,y)=\{y\}$ then $\gamma(h^p,x,y)$ is trivial.

\item If $\C O(h_t,y)$ is a contractible periodic orbit bounding a disc containing the point $x$
then $\gamma(h^p,x,y)$ is an integer power of $\sigma^2$.

\item If $\C O(h_t,y)$ is a contractible periodic orbit bounding a disc not containing
the point $x$ then $\gamma(h^p,x,y)$ is either trivial or equal to $\sigma^2$.

\item If $\C O(h_t,y)$ is a homotopically nontrivial periodic orbit then
its image is a simple closed curve. There exists a symplectic diffeomorphism
of the torus $f\in \OP{Symp}(\B T^2)$ preserving the basepoints
$z_1$ and $z_2$ such that the image $f(\C O(h_t,y))$ of the
orbit represents the standard generator $(1,0)\in \B Z^2=\pi_1(\B T^2)$
disjoint from $f(x)$. Thus the braid $\gamma(fh^pf^{-1},f(x),f(y)))=(F,A)$, where
both $F\in \B F_2$ and $A\in \B Z^2$ are powers of primitive elements.
Observe that $\gamma(fh^pf^{-1},f(x),f(y)))=f_*(\gamma(h^p,x,y))$, where
$f_*\colon \B P_2(\B T^2)\to \B P_2(\B T^2)$ is the automorphism induced
by~$f$. Since $f$ induces an automorphism of the
quotient $\B F_2$ we get that $\pi(\gamma(h^p,x,y))$ is a power of a primitive element.

\item
If the orbit $\C O(h_t,y)$ is nonperiodic then there exists $p_0\in \B N$
such that
$$
\#\{\gamma(h^p,x,y)\in \B P_2(\B T^2)\,|\,p\geq p_0\,\}\leq 2.
$$
Indeed, let $y_+:=\lim_{t\to \infty}h_t(y)$ be the limit point and let
$\epsilon >0$. There exists $p_0$ such that $|h^p(y)-y_+|<\epsilon$
for every $p\geq p_0$. Depending on a relative position of the
points $z_1,z_2,x$ and $h^p(y)$ the braids $\gamma(h^p,x,y)$
and $\gamma(h^q,x,y)$ for $p,q\geq p_0$ may differ by at most one
crossing arising when the endpoints $h^p(y)$ or $h^q(y)$ and $x$
are joined to the basepoints.
\end{enumerate}

\item The orbit $\C O(h_t,x)$ is nonperiodic.
Let $x_{+}:=\lim_{t\to \infty}h_t(x)$.
\begin{enumerate}
\item
If $\C O(h_t,y)$ is either constant or nonperiodic then
$$
\#\{\gamma(h^p,x,y)\in \B P_2(\B T^2)\,|\,p\geq p_0\,\}\leq 2
$$
as in the previous case. The only difference is that for given
$\epsilon$ one has to choose $p_0$ such that both
$|h^p(x)-x_+|< \epsilon$ and $|h^p(y)-y_+|<\epsilon$ for all $p\geq p_0$.

\item
If $\C O(h_t,y)$ is a contractible periodic orbit such that
$\C O(h_t,x)$ is contained in the disc bounded by $\C O(h_t,y)$
then $\gamma(h^p,x,y)$ is a power of $\sigma^2$.

\item
If $\C O(h_t,y)$ is a contractible periodic orbit such that
$\C O(h_t,x)$ is not contained in the disc bounded by $\C O(h_t,y)$
then $\gamma(h^p,x,y)$ is either trivial or equal to $\sigma^2$.

\item
If $\C O(h_t,y)$ is a homotopically nontrivial periodic orbit
then, as in the case (1)(d) above, we get that $\gamma(h^p,x,y)$
is a power of a primitive element.
\end{enumerate}

\item
The orbit $\C O(h_t,x)$ is contractible periodic.

\begin{enumerate}
\item The case when $\C O(h_t,y)$ is either constant or nonperiodic
has been dealt with above.
\item If the orbits $\C O(h_t,x)$ and $\C O(h_t,y)$ are concentric
then $\gamma(h^p,x,y)$ is a power of $\sigma^2$.
\item If the orbits $\C O(h_t,x)$ and $\C O(h_t,y)$ are contractible
and not concentric then $\gamma(h^p,x,y)$ is either trivial or
equal to $\sigma^2$.
\item If $\C O(h_t,y)$ is periodic and homotopically nontrivial
then the braid $\gamma(h^p,x,y)$ is a power of a primitive element and the
argument is the same as in the analogous cases above.
\end{enumerate}

\item
If $\C O(h_t,x)$  is periodic and homotopically nontrivial
then the only case which has not been done above is when
the orbit $\C O(h_t,y)$ is periodic and not contractible.
In this case the images of both orbits are disjoint simple
closed curves and thus there exists a symplectic diffeomorphism
$f\in \OP{Symp}(\B T^2)$ preserving basepoints $z_1$ and $z_2$
such that both $f(\C O(h_t,x))$ and
$f(\C O(h_t,y))$ are disjoint simple closed curves representing
the generator $(1,0)\in H_1(\B T^2;\B Z)$ (recall that the intersection form on $H_1(\B T^2;\B Z)$ is non-degenerate and anti-symmetric).
In this case we have
$$
\gamma(fh^pf^{-1},f(x),f(y))=
\begin{cases}
b_i^m \,b_j^n\\
b_j^m\, b_i^n\, \sigma^2,
\end{cases}
$$
for some $m,n\in \B Z$, where $i, j\in \{1,2\}$ are distinct.
To see this recall that the conjugation by $\sigma$ swaps the
generators $b_1$ and $b_2$. It may also be useful to use the
following computation
$$
\sigma b_i^mb_j^n \sigma =
\sigma b_i^m b_j^n \sigma^{-1}\sigma^2
= b_j^m b_i^n \sigma^2.
$$
Thus the image of the above braid in the free group is equal to
either $b^{m-n}$ or $b^{m-n}[a,b]$.
Thus the image of the braid $\gamma(h^p,x,y)$ in the free group
is a product of a power of a primitive element and a commutator
of two primitive elements.
\end{enumerate}
According to a theorem of Nielsen \cite{MR1511907}, the commutator of
two primitive elements is conjugate to $[a,b]^{\pm 1}$.

As a conclusion we obtain that the projection
$\pi(\gamma(h^p,x,y))$ is equal to either one of the following:
\begin{itemize}
\item an integer power of the commutator $[a,b]$,
\item an integer power of a primitive element,
\item a product of a power of a primitive element and a conjugate
of the commutator $[a,b]$ or its inverse,
\end{itemize}
or there exists $p_0\in \B N$ such that
$$
\#\{\gamma(h^p,x,y)\in \B P_2(\B T^2)\,|\,p\geq p_0\,\}\leq 2.
$$
Let $\psi\colon \B F_2\to \B R$ be a homogeneous quasimorphism
vanishing on primitive elements and on the commutator $[a,b]$.
If $\gamma(h^p,x,y)$ attains finitely many
values for $p\geq p_0$ then
$$
\lim_{p\to \infty}\frac{1}{p}\psi(\gamma(h^p,x,y))=0
$$
and hence $\C G(\psi)(h)=0$.
If $\pi(\gamma(h^p,x,y))$ is a power of either a primitive element or
a conjugate of the commutator $[a,b]$  then
$\psi(\pi(\gamma(h^p,x,y)))=0$ by the hypothesis and we also get that
$\C G(\psi)(h)=0$.
Finally, if $\pi(\gamma(h^p,x,y))$ is a product
of a power of a primitive element and a power of a conjugate of the
commutator $[a,b]$ then
$$
|\C G(\psi)(h)|\leq D_{\psi}.
$$
Since for every autonomous diffeomorphism $h$ and $n\in \B N$, $h^n$ is also autonomous we get
$$
|\C G(\psi)(h)|=\left|\frac{\C G(\psi)(h^n)}{n}\right|\leq\frac{D_\psi}{n}.
$$
This concludes the proof of the vanishing property of $\C G(\psi)$ on the
set of autonomous diffeomorphisms.
\end{proof}

\subsection{Palindromes and primitive elements}\label{SS:bardakov}
The following observation and its proof are due to Bardakov, Shpilrain and Tolstykh
\cite{MR2125453}.
\begin{lemma}\label{L:bardakov}
Every primitive element $w\in \B F_2$ of the free group of rank $2$ is
a product of up to two palindromes.
\end{lemma}
\begin{proof}
Let $\sigma\colon \B F_2\to \B F_2$ be an automorphism defined
on generators by $\sigma(a)=a^{-1}$ and
$\sigma(b)=b^{-1}$. Consider the extension
$$
\B F_2\stackrel{\iota}\to \OP{Aut}(\B F_2)
\stackrel{\pi}\to \OP{Out}(\B F_2)=\OP{GL}(2,\B Z),
$$
where the quotient is identified with the automorphism group
of the abelianisation $\B Z^2$ of the free group $\B F_2$.

Let $\theta\in \OP{Aut}(\B F_2)$ be an automorphism. Since
the image $\pi(\sigma)\in \OP{Out}(\B F_2)$ is equal to
$-\OP{Id}$, the image $\pi[\sigma,\theta]$ of the commutator
is trivial. This implies that
$[\sigma,\theta]=I_{p(\theta)}$ is an inner automorphism
for some $p(\theta)\in \B F_2$.
For example, if $\tau\in \OP{Aut}(\B F_2)$ is defined
by $\tau(a)={ab}$ and $\tau(b)=b$
then $p(\tau)=a$.

The following computation proves that $\sigma(p(\theta))=p(\theta)^{-1}$
which means that  $p(\theta)$ is a palindrome (notice that $\sigma$ is an
involution).
\begin{align*}
I_{\sigma(p(\theta))}&=\sigma\,I_{p(\theta)}\,\sigma^{-1}
=\sigma[\sigma,\theta]\sigma^{-1}\\
&=[\theta,\sigma]
=[\sigma,\theta]^{-1}
=I_{p(\theta)^{-1}}.
\end{align*}

The second observation is that
$$
p(\theta\xi)=p(\theta)\theta(p(\xi))
$$
for any $\theta,\xi\in \OP{Aut}(\B F_2)$. Indeed,
\begin{align*}
I_{p(\theta\xi)}&=[\sigma,\theta\xi]\\
&=\sigma\theta\sigma\theta^{-1}\cdot\theta\sigma\xi\sigma\xi^{-1}\theta^{-1}\\
&=[\sigma,\theta]\theta[\sigma,\xi]\theta^{-1}\\
&=I_{p(\theta)}\theta I_{p(\xi)}\theta^{-1}\\
&=I_{p(\theta)}I_{\theta(p(\xi))}=I_{p(\theta)\theta(p(\xi))}.
\end{align*}
Evaluating this identity on the automorphism $\tau$ defined above we get that
\begin{align*}
p(\theta\tau)&=p(\theta)\theta(p(\tau))\\
p(\theta\tau)&=p(\theta)\theta(a)\\
p(\theta)^{-1}p(\theta\tau)&=\theta(a),
\end{align*}
which shows that any primitive element $\theta(a)$ is a product
of two palindromes.

\end{proof}
\begin{corollary}\label{C:vanish-primitive}
If $\psi\colon \B F_2\to \B R$ is a homogeneous quasimorphism which
vanishes on palindromes then it vanishes on primitive elements.
\end{corollary}

\begin{proof}
Let $w\in \B F_2$ be a primitive element. It follows from Lemma
\ref{L:bardakov} that $w=uv$, where $u,v\in \B F_2$ are palindromes.
If $n\in \B N$ is a positive integer then
$$
w^n = (uv)^n = \left((uv)^{n-1}u\right) v,
$$
which shows that $w^n$ is a product of two palindromes. It implies that
$|\psi(w^n)|\leq D_{\psi}$. Since $\psi$ is homogeneous we get that
it vanishes on primitive elements.
\end{proof}

\begin{corollary}\label{C:autonomous}
If $\psi\in \B Q(\B F_2;\B Z/2\times \B Z/2)$ then the quasimorphism
$\C G(\pi^*\psi)\in \B Q(\OP{Ham}(\B T^2)$ vanishes on autonomous diffeomorphisms.
\end{corollary}
\begin{proof}
Let $\sigma=\sigma_a \sigma_b \in \OP{Aut}(\B F_2)$. It acts on words
by inverting all letters. In particular, an element $w\in \B F_2$
is a palindrome if and only if $\sigma(w) = w^{-1}$.
If $\psi \in \B Q(\B F_2;\B Z/2\times \B Z/2)$ then $\psi$ is
invariant under the action of $\sigma$:
$$
\psi(\sigma(w)) = \psi(w).
$$
If $w\in \B F_2$ is a palindrome then $\psi(w^{-1}) = \psi(\sigma(w))=\psi(w)$
and it follows from the homogeneity of $\psi$ that it vanishes
on palindromes. It is then a consequence of Corollary \ref{C:vanish-primitive}
that $\psi$ vanishes on primitive elements.

As explained in the computation on page \pageref{Eq:commutator}, $\psi$ vanishes
on the commutator $[a,b]$ of the generators of $\B F_2$ and hence,
according to Lemma~\ref{L:autonomous}, the quasimorphism
$\C G(\pi^*\psi)$ vanishes on autonomous diffeomorphisms.
\end{proof}

\section{Further results}
\subsection{A lift to symplectic diffeomorphisms}
The braid $\gamma(h,x,y)$
is not well defined for $h\in \OP{Symp}_0(\B T^2)$. That is, it
depends on the isotopy from the identity to $h$.
For example, the isotopy defined $h_t(u,v)=(u+t,v)$
is a loop based at the identity and $\gamma(h_t,x,y)=a_1a_2$.

\begin{lemma}\label{L:loop}
Let $\ell\colon [0,1]\to \OP{Symp}_0(\B T^2)$ be a loop based at the identity.
Then for every $x,y\in \B T^2$ the braid $\gamma(\ell,x,y)$ is central.
\end{lemma}

\begin{proof}
Since the inclusion  $\B T^2\to \OP{Symp}_0(\B T^2)$ (where the torus acts on
itself by translations) is a homotopy equivalence, the loop $\ell$ is isotopic
to a concatenation of loops $h_a\colon(u,v)\mapsto (u+t,v)$ and
$h_b\colon(u,v)\mapsto (u,v+s)$. This implies that
$$
\gamma(\ell,x,y) = a_1^m a_2^m b_1^n b_2^n\in \B P_2(\B T^2),
$$
for some $m,n\in \B Z$. Observe that the center of $\B P_2(\B T^2)$
is isomorphic to $\B Z^2$ and is generated by $a_1a_2$ and $b_1b_2$
(see Lemma \ref{L:P2(T)}.
\end{proof}

It follows from the lemma that if $\psi\in \B Q(\B P_2(\B T^2))$
vanishes on the center then $\psi(\gamma(h,x,y))$ is well
defined for $h\in \OP{Symp}_0(\B T^2)$. We thus have a lift
of the Gambaudo-Ghys homomorphism
$$
\widehat{\C G}\colon \B Q(\B P_2(\B T^2),{\tt CENTER})\to
\B Q(\OP{Symp}_0(\B T^2)).
$$
\begin{lemma}\label{L:centre}
The composition
$$
\B Q(\B B_2(\B T^2),{\sc{\tt CENTER}})\stackrel{\iota^*}\to
\B Q(\B P_2(\B T^2),{\sc{\tt CENTER}})\stackrel{\widehat{\C G}}\to
\B Q(\OP{Symp}_0(\B T^2))
$$
is injective.
\end{lemma}
The proof of this lemma is essentially the same as the proof of
Theorem 2 in \cite{MR3391653}. If $\psi\in \B Q(\B F_2)$ then
the composition $\psi\circ \pi$ vanishes on the center and hence
every quasimorphism on $\B F_2$ yields a quasimorphism on the
group of symplectic diffeomorphisms. In order to ensure that
it is nontrivial we require that $\psi\circ \pi$ extends
to the full braid group and this holds if the quasimorphism
$\psi$ vanishes on palindromes (see Proposition~\ref{P:palindromes}).
\begin{corollary}\label{C:palindromes}
The composition
$$
\B Q(\B F_2,{\tt PAL})\stackrel{\pi^*}\to \B Q(\B P_2(\B T^2))
\stackrel{\widehat{\C G}}\to \B Q(\OP{Symp}_0(\B T^2))
$$
is injective.\qed
\end{corollary}
Since $\OP{Ham}(\B T^2)$ is equal to the commutator subgroup
of $\OP{Symp}_0(\B T^2)$ the kernel of the homomorphism
$\B Q(\OP{Symp}_0(\B T^2))\to \B Q(\OP{Ham}(\B T^2))$ induced
by the inclusion consists of homomorphisms. Since
$\B Q(\B F_2,{\tt PAL})$ contains no homomorphism we get that
the composition
$$
\B Q(\B F_2,{\tt PAL})\to \B Q(\OP{Ham}(\B T^2))
$$
is injective.

\begin{remark}
This is a slightly stronger statement than Proposition~\ref{P:injectivity} and
it could serve as an alternative part of the proof of the main
Theorem~\ref{T:main}.  We chose a more direct approach in order to have a
complete proof which makes the paper  selfcontained and also because the proof
of Proposition \ref{P:injectivity} provides explicit examples of
diffeomorphisms. On the other hand, the above arguments allow us
to provide examples of Calabi quasimorphisms which are presented next.
\end{remark}

\subsection{The Calabi property and continuity}\label{SS:snake}
Let $(M,\omega)$ be a symplectic manifold and let $B\subset M$ is a
displaceable symplectic ball. A homogeneous quasimorphism
$\Psi\colon \OP{Ham}(M,\omega)\to \B R$ is called {\em Calabi} (or has the
Calabi property) if its restriction to a subgroup $\OP{Ham}(B)$ coincides with
the Calabi homomorphism.  The definition is due to Entov and Polterovich
\cite{MR1979584} who constructed first examples of Calabi quasimorphisms using
quantum homology. Their examples include the sphere $\B S^2$ and they asked
whether there were Calabi quasimorphisms for other surfaces.  Pierre Py gave a
positive answer to this question in \cite{MR2224660, MR2253051}.

Here we provide a Calabi quasimorphism by producing a slightly modified example.

\begin{example}[The snake quasimorphism]\label{E:snake}
Let $w\in \B F_2$ be an element. It defines a path on the plane starting at
the origin, consisting of horizontal and vertical segments of integer
length with turning points on the integer lattice.
See Figure \ref{F:snake} for an example.

\begin{figure}[h]
\begin{tikzpicture}[line width=5pt, scale=0.26]
\draw[step=100pt, line width=1pt] (-18,-10) grid (18,16);
\draw[green, ->, line width=2pt]
(0,-12) -- (0,18);
\draw[green, ->, line width=2pt]
(-20,0) -- (20,0);
\draw[blue]
(0,0) -- (7,0) -- (7,-3.5) -- (14,-3.5) -- (14,14)
-- (3.5,14) -- (3.5,-7) -- (-7,-7) -- (-7,7) -- (7,7)
-- (7,10.5);
\end{tikzpicture}
\caption{The element $w=a^2b^{-1}a^2b^{5}a^{-3}b^{-6}a^{-3}b^{4}a^4b$}
\label{F:snake}
\end{figure}

Let ${\xi}\colon \B F_2\to \B Z$
be defined by ${\xi}(w):=  L(w) -  R(w),$
where $L(w)$ and $R(w)$ denote the number of the left and right turns of
the path defined by $w$.
Thus the value at an element drawn in Figure \ref{F:snake} is
${\xi}(w) = 5 - 4 =1$.

If $w\in \B F_2$ is a palindrome then the induced path is symmetric with
respect to the half turn about its mid point and hence the initial turns
become the opposite terminal turns and hence they cancel. Thus
${\xi}$ vanishes on palindromes and hence its homogenization vanishes
both on palindromes and primitive elements.
Since ${\xi}\left([a,b]^n\right)=4n-1$ we see that ${\xi}$
is unbounded. Hence its homogenization $\widehat{\xi}$ is nontrivial and
$\widehat{\xi}([a,b])=4$. Thus the quasimorphism
$\C G(\pi^*\xi)$ is nontrivial and has the Calabi property.
\hfill $\diamondsuit$
\end{example}

\subsection*{Acknowledgments} We would like to thank Luis Haug for fruitful discussions
and Pierre Py for comments on the first version of the paper.
We thank the CRM and ISM who supported the visit of K\k{e}dra in Montreal.
K\k{e}dra  also thanks the Max Planck Institute for Mathematics for supporting
his visit in Bonn.

Part of this work has been done during Brandenbursky's stay
at IHES and CRM. We wish to express his gratitude to both institutes.
Brandenbursky was supported by CRM-ISM fellowship and NSF grant No. 1002477

This work has been done during Shelukhin's stay in CRM, ICJ Lyon 1, Institut Mittag Leffler,
and IAS. He thanks these institutions for their warm hospitality. He was partially supported by CRM-ISM fellowship,
ERC Grant RealUMan, Mittag Leffler fellowship, and NSF grant No. DMS-1128155.

\bibliography{bibliography}
\bibliographystyle{plain}

\end{document}